\documentclass{Tran-l}

\usepackage{amssymb}
\usepackage{amsmath}

\usepackage{tikz,tikz-cd}

\usepackage[colorlinks=true,linktocpage=true,linkcolor=blue,citecolor=red]{hyperref}

\newtheorem{theorem}{Theorem}[section]

\newtheorem{definition}[theorem]{Definition}

\newtheorem{lemma}[theorem]{Lemma}

\newtheorem{proposition}[theorem]{Proposition}

\theoremstyle{definition}
\newtheorem{example}[theorem]{Example}
\newtheorem{construction}[theorem]{Construction}
\newtheorem{remark}[theorem]{Remark}

\begin{document}

\title{Compact Hausdorff locales in presheaf toposes}

\author{Simon Henry}\email{Shenry2@uottawa.ca}
\author{Christopher Townsend}\email{info@christophertownsend.org}

\maketitle
\begin{abstract}
We prove that for any small category $\mathcal{C}$, the category  $\mathbf{KHausLoc}_{\hat{\mathcal{C}}}$ of compact Hausdorff locales in the presheaf topos $\hat{\mathcal{C}}$, is equivalent to the category of functors $\mathcal{C} \to \mathbf{KHausLoc}$.  
\end{abstract}

\section{Introduction}
In this paper we prove for any small category $\mathcal{C}$ that there is an equivalence of categories:
\begin{eqnarray*}
\mathbf{KRegFrm}_{\hat{\mathcal{C}}} \simeq [\mathcal{C}^{op},\mathbf{KRegFrm}  ]
\end{eqnarray*}
where $\mathbf{KRegFrm}$ is the category of compact regular frames, $\hat{\mathcal{C}}$ is the presheaf topos $[ \mathcal{C}^{op}, \mathbf{Set}]$ and $\mathbf{KRegFrm}_{\hat{\mathcal{C}}}$ is the category of compact regular frames in the topos $\hat{\mathcal{C}}$. Since $\mathbf{KRegFrm}$ is dual to the category of compact Hausdorff locales ($\mathbf{KHausLoc}$) in every topos, the claim of the abstract is shown with this categorical equivalence.

This result can be thought of as a new example of ``open/proper duality'' (e.g. \cite{parallel}). Indeed, discrete locales are locales $X$ such that both the unique map $X \to 1$ and the diagonal map $X \to X \times X$ are open maps, and discrete locales in the topos $\hat{\mathcal{C}}$ correspond to presheaves on $\mathcal{C}$; that is, to functors $\mathcal{C}^{op} \to \mathbf{Set}$. In this note, we are proving that ``dually'', compact Hausdorff locales in $\hat{\mathcal{C}}$, that is locales $X$ in $\hat{\mathcal{C}}$ such that both the unique map $X \to 1$ and the diagonal map $X \to X \times X$ are proper, correspond to functors $\mathcal{C} \to \mathbf{KHausLoc}$.

In summary, the proof proceeds as follows. In section \ref{sec:NDL} we show that compact regular frames can be characterised as completions of normal distributive lattices, with the completion given by an idempotent functor $C$ acting on the category of normal distributive lattices ($\mathbf{NDL}$). This characterisation applies internally in the presheaf topos $\hat{\mathcal{C}}$, hence we have established that any compact regular frame in $\widehat{\mathcal{C}}$ is the completion of an object of $\mathbf{NDL}_{\hat{\mathcal{C}}}$. Because the notion of normal distributive lattice is geometric (in the sense of geometric logic, e.g. D1 of \cite{Elephant}), we have an isomorphism of categories $\mathbf{NDL}_{\hat{\mathcal{C}}} \cong [\mathcal{C}^{op},\mathbf{NDL}]$. These observations give us a way of understanding the category $\mathbf{KRegFrm}_{\hat{\mathcal{C}}}$ as consisting of objects obtained by completing objects of $[\mathcal{C}^{op},\mathbf{NDL}]$. To conclude the proof we need an explicit description of the relative version of the idempotent endofunctor $C$, acting on $\mathbf{NDL}_{\hat{\mathcal{C}}}$ and essentially this is what is provided by the rest of the paper.

In section \ref{sec:lax} we introduce a categorical construction on presheaves taking values in an order enriched category. This construction is adjoint to the forgetful functor from the category of presheaves and natural transformation to the category of presheaves and lax natural transformation. In section \ref{sec:exemple}, we give examples of this construction and explain how it is used to describe the relative version of the construction $C$ mentioned above. Finally in section \ref{sec:main_th} we put everything together and prove the main theorem.

\section{Normal distributive lattices}
\label{sec:NDL}

By a \emph{distributive lattice}, we mean a poset with finite\footnote{Including nullary} joins and finite meets which satisfies the distributivity law  $a \wedge ( b \vee c) = (a \wedge b) \vee (a \wedge c)$, or equivalently the dual distributivity law  $a \vee ( b \wedge c) = (a \vee b) \wedge (a \vee c)$. Morphisms of distributive lattices are order preserving maps that preserves finite joins and finite meets (i.e. lattice homomorphisms).

In a poset $P$, if $I \subset P$, we write $\downarrow I = \{ a \in P | \exists i \in I, a \leqslant i \}$, and $\downarrow a = \downarrow \{a\}$ if $ a \in P$. If $f : X \to Y$ is a function we denote by $\exists_f$ the direct image function from the power set of $X$ to the power set of $Y$.

A subset $I \subseteq D$ of a distributive lattice $D$ is an \emph{ideal} if and only if $a \leq b \in I \Rightarrow a \in I$ and $I$ is closed under finite joins. The set of all ideals, written $idl(D)$, is itself a distributive lattice.

\begin{definition}
A distributive lattice $D$ is \emph{normal} provided for any $a,b \in D$ if $ a \vee b = 1$ then there exists $a',b' \in D$ such that
\begin{eqnarray*}
a' \wedge b' = 0 \text{ and } a' \vee b = 1 = a \vee b'  \text{.}
\end{eqnarray*}
We denote by $\mathbf{NDL}$ the full subcategory of distributive lattices defined by this condition.
\end{definition}

 As usual, one defines the relation $a \triangleleft b$ by $\exists c$ such that $a \wedge c = 0$ and $b \vee c = 1$. An equivalent way to define normal is then $a \vee b =1 \Rightarrow \exists a' \triangleleft a$ such that $a' \vee b = 1$. It also follows that the relation $\triangleleft$ is interpolative in a normal distributive lattice: say $ a \triangleleft b$, witnessed by $c$, so that $b \vee c =1$; then there exists $b' \triangleleft b$ and $b' \vee c =1$ so that $c$ also witnesses $ a \triangleleft b'$.

\begin{example}
A compact regular frame $\mathcal{O}X$ is normal. Regularity is the assertion that $b = \bigvee \{ b' | b' \triangleleft b \}$ for every open $b$. But the join is directed so $a \vee b = 1$ implies there exists $b' \triangleleft b$ with $a \vee b' =1$ by compactness,  which as observed above implies that $\mathcal{O}X$ is normal.  Frame homomorphisms are lattice homomorphism so there is a forgetful functor $u:\mathbf{KRegFrm} \to \mathbf{NDL}$.
\end{example}

We now have a couple of lattice theoretic propositions which show that compact regular frames can be seen as completions of normal distributive lattices.  

\begin{proposition}\label{prop:C_completion}
Let $N$ be a normal distributive lattice. Define
\begin{eqnarray*}
C(N) = \{ I \subseteq N | I \text{an ideal and } \forall a \in I \text{ }\exists b \in I \text{ such that } a \triangleleft b \} \text{.}
\end{eqnarray*}
Then:

1. $C(N)$ is a distributive lattice.

2. $\Downarrow: N \to C(N)$, defined by $ \Downarrow a = \{ b | b \triangleleft a \}$, is a well defined lattice homomorphism. 

3. $C(N)$ is a normal distributive lattice. 

4. By setting $C(f)(I) = \downarrow \exists_f (I)$ for any lattice homomorphism $f:N \to M$ we have defined an order enriched functor $C: \mathbf{NDL} \to \mathbf{NDL}$.

\end{proposition}

\begin{proof}

1. In fact $C(N)$ is a sublattice of $idl(N)$. The only slightly non-trivial observation is closure under binary joins. Recall that if $I$ and $J$ are ideals then their join in $idl(N)$ is given by $\downarrow \{ a_1 \vee a_2 | a_1 \in I, a_2 \in J \}$. To check that this join is in C(N) if $I$ and $J$ are, recall that if $a'_i$ witnesses that $a_i \triangleleft b_i$ for $i = 1,2$ then $a'_1 \wedge a'_2$ witnesses that $a_1 \vee a_2 \triangleleft b_1 \vee b_2$. It follows that $C(N)$ is a distributive lattice.

2. This is well defined as $\triangleleft$ is interpolative. That it is a lattice homomorphism follows from our description of meets and joins in $C(N)$; for binary joins, observe that if $b \triangleleft a_1  \vee  a_2 $, witnessed by $c$ with $a_1 \vee a_2 \vee c = 1$, then by double application of the equivalent way of defining normal there exists $a'_1,a'_2$ such that $a'_1 \vee a'_2 \vee c = 1$, from which $b\leq a'_1 \vee a'_1$. 

3. Say $I \vee J = 1$. Then there exists $a \in I, b \in J$ such that $a \vee b =1$. By normality of $N$ there exists $a' \triangleleft a$ such that $a' \vee b =1$. Because, by normality again, there exists $a'' \triangleleft a'$ with $a'' \vee b =1$, we have that $\Downarrow a' \vee J = 1$. Say $c$ witnesses $a' \triangleleft a$, then $\Downarrow a' \wedge \Downarrow c \subseteq \downarrow a' \wedge c = 0$. But $c \vee a = 1$ implies there exist $c' \triangleleft c$ with $c' \vee a = 1$. It follows that $\Downarrow c \vee I =1$ as $a \in I$. Therefore $\Downarrow c$ witnesses that $\Downarrow a' \triangleleft I$ which completes the proof that $C(N)$ is normal.  

4. That $C(f)$ is well defined is clear as $f$ will preserve $\triangleleft$. It is a routine exercise to show that because $f$ preserves finitary meets and joins, $C(f)$ must be a lattice homomorphism.
\end{proof}

Our final Proposition for this section shows that the category of compact regular frames is a splitting of the category of normal distributive lattices. 
\begin{proposition}\label{prop:C_compact}

For every normal distributive lattice $N$, $C(N)$ is a compact regular frame and $C(f):C(N) \to C(M)$ is a frame homomorphism for every lattice homomorphism $f: N \to M$. Therefore $C$ determines a functor $c:\mathbf{NDL} \to \mathbf{KRegFrm}$ such that $C = uc$ and $cu \cong Id_{\mathbf{KRegFrm}}$, where $u: \mathbf{KRegFrm} \to\mathbf{NDL} $ is the forgetful functor. That is, $\mathbf{KRegFrm}$ is a splitting of $C:\mathbf{NDL}\to \mathbf{NDL}$.

\end{proposition}
\begin{proof}
Certainly $C(N)$ is a frame as directed joins are given by union, from which it is trivial that $C(N)$ is compact. By definition of $C(N)$, any $I \in C(N)$ satisfies $I = \bigcup^{\uparrow} \{ \Downarrow a | a \in I \}$. But just as in the previous proposition we can show $\Downarrow a \triangleleft I$ for every $ a \in I$ so $C(N)$ is regular. (Explicitly, if $a \in I$ then $\exists b \in I$ with $a \triangleleft b$, witnessed by $c$, say. So certainly $\Downarrow a \wedge \Downarrow c = 0$. But $c \vee b =1$, so there exists $c' \triangleleft c$ such that $c' \vee b = 1$ and so since $c' \vee b \in \Downarrow c \vee I$ we have $\Downarrow a \vee I = 1$. Therefore $\Downarrow c$ witnesses that $\Downarrow a \triangleleft I$.)

As directed joins are given by union and union clearly commutes with $\downarrow \exists_f$, it can be verified that $C(f)$ preserves directed joins. Since we have already established that $C(f)$ is a lattice homomorphism we can conclude that $C$ determines a functor $c:\mathbf{NDL} \to \mathbf{KRegFrm}$, and certainly $C=uc$ by construction.

Next, say $N= \mathcal{O}X$, a compact regular frame. Send $I \in C\mathcal{O}X$ to $\bigvee^{\uparrow} I \in \mathcal{O}X$. To show $\bigvee^{\uparrow} = \Downarrow^{-1}$, one way round is clear by regularity of $\mathcal{O}X$. For the other way round, firstly if $a \in I $ then there exists $a' \in I $ with $ a \triangleleft a'$; but then $ a \triangleleft \bigvee^{\uparrow} I$. Secondly if $ a \triangleleft \bigvee^{\uparrow} I$, then $ a \ll \bigvee^{\uparrow} I$ as $\mathcal{O}X$ is compact regular; hence $ a \in I$. (Recall that in a compact regular frame $\ll = \triangleleft$; see e.g. VII Lemma 3.5 of \cite{Stone}.) It follows that $cu \cong Id_{\mathbf{KRegFrm}}$ with the naturality aspect being clear as the isomorphism is directed join and frame homomorphisms are directed join preserving. 
\end{proof}

All the results of this section can be applied internally in any topos $\mathcal{T}$, showing that the category $\mathbf{KRegFrm}_\mathcal{T}$ of compact regular frames in the topos $\mathcal{T}$ is a splitting of the category $\mathbf{NDL}_\mathcal{T}$. We will use this to understand the category $\mathbf{KRegFrm}_{\hat{\mathcal{C}}}$, but in order to do so we will need an explicit description of the functor $C_{\hat{\mathcal{C}}} : \mathbf{NDL}_{\hat{\mathcal{C}}} \to \mathbf{NDL}_{\hat{\mathcal{C}}}$; the next two sections lead up to this explicit description.

\section{Turning lax natural transformation into ordinary natural transformations}
\label{sec:lax}

In this section we outline a basic categorical construction that relates lax natural transformations between presheaves taking values in an order enriched category, to ordinary natural transformations. In the next section we will show that this construction is closely related to the functor $C_{\hat{\mathcal{C}}}$.

We will work with order enriched categories; that is,  homsets are partially ordered sets (posets) and composition preserves the order. The order relation between morphisms will be denoted $ \sqsubseteq $. Universal properties are required to establish order isomorphisms (not just bijections) between the posets of morphisms.  

If $F_1,F_2:\mathcal{C}^{op} \to \mathfrak{K}$ are two order enriched functors between order enriched categories, then a \emph{lax} natural transformation $\phi: F_1 \xrightarrow{\sqsubseteq} F_2$ is a collection of morphisms $\phi_a:F_1(a) \to F_2(a)$ indexed by objects $a$ of $\mathcal{C}$ such that for any morphism $h:a \to a'$ of $\mathcal{C}$, $\phi_{a'}F_1(h) \sqsupseteq F_2(h) \phi_a$; i.e.
\[\begin{tikzcd}[row sep = large]
  F_1(a) \ar[r,"\phi_a"] \ar[d,"F_1(h)"swap] \ar[dr,phantom,"\sqsupseteq"description]& F_2(a) \ar[d,"F_2(h)"] \\
  F_1(a') \ar[r,"\phi_{a'}"] & F_2(a') 
\end{tikzcd} \]

We use $[\mathcal{C}^{op},\mathfrak{K}]^{\sqsubseteq}$ as notation for the order enriched category of presheaves with lax natural transformations between them. The ordering on the lax natural transformation is pointwise.

Recall that a lax limit of an order enriched functor functor $D:\mathcal{J} \to \mathfrak{K}$ is a universal lax cone, where a \emph{lax} cone is  collection of morphisms $\pi_j:lim^{\sqsubseteq}_{\mathcal{J}} D \to D(j)$ indexed by object $j$ of $\mathcal{J}$ such that for any morphism $\alpha:i \to j$ of $\mathcal{J}$,  $D(\alpha) \pi_i \sqsubseteq \pi_j$. 
\begin{example}
The order enriched category of posets, $\mathbf{Pos}$, has arbitrary lax limits. Given $D:\mathcal{J} \to \mathbf{Pos}$,
\begin{eqnarray*}
lim^{\sqsubseteq}_{\mathcal{J}} D = \{ (x_j) \in \prod_{j \in Ob ( \mathcal{J} )} D(j) |D(\alpha)x_i \leq x_j \text{ } \forall \alpha: i \to j \in \mathcal{J} \} \text{.} 
\end{eqnarray*}

Another example is the category of distributive lattices; it is easy to check that lax limits of distributive lattices are created in $\mathbf{Pos}$
\end{example}

\begin{example}\label{sup_example}
The category of suplattices (i.e. complete lattices with arbitrary join preserving maps as morphisms), $\mathbf{Sup}$, has arbitrary lax limits. They are created in $\mathbf{Pos}$ with join given pointwise (i.e. $\bigvee_{i \in I} (x^i_j) = (\bigvee_{i \in I} x^i_j)$ ).
\end{example}

We will only need the existence of lax limits for $\mathcal{J}$ with an initial object. 
\begin{definition}
An order enriched category $\frak{K}$ is \emph{initial-lax complete} if it has a lax limit $lim^{\sqsubseteq}_{\mathcal{J}} D$ whenever $\mathcal{J}$ has an initial object. 
\end{definition}

The category $\mathbf{NDL}$ is relatively far from being complete in general, but it does satisfies this condition:

\begin{proposition}
The order enriched category $\mathbf{NDL}$ is initial-lax complete with initial lax limits being created in $\mathbf{Pos}$. 
\end{proposition}
\begin{proof}
Consider a diagram $D:\mathcal{J} \to \mathbf{NDL}$, such that $\mathcal{J}$ has an initial object. For any object $j$ of $\mathcal{J}$, write $!^j:0 \to j$ for the unique map to $j$. We have commented already that the category of distributive lattices is lax complete. Since a morphism of normal distributive lattices is the same thing as a morphism of distributive lattices, we must just check that the distributive lattice 
\begin{eqnarray*}
N = \{ (a_j) \in \prod_j D(j) | D(f)a_i \leq a_j \text{  }  \forall f: i \to j  \in  \mathcal{J} \}
\end{eqnarray*}
is normal. Say $(a_j) \vee (b_j) = 1_N$. Then $a_0 \vee b_0 = 1_{D0}$. So as $D0$ is normal there exists $a'_0$ and $b'_0$ such that $a'_0 \wedge b'_0 = 0_{D0} \text{ and } a'_0 \vee b_0 = 1_{D0} = a_0 \vee b'_0$. Let $a_j' = D(!^j)(a_0')$ and $b_j' = D(!^j)(b_0')$. Then $(a'_j),(b'_j) \in N$ and $(a'_j) \wedge (b'_j) = 0_N,  (a'_j) \vee (b_j) = 1_N = (a_j) \vee (b'_j)$, the last because $D(!^j)(b_0) \leq b_j$ and $D(!^j)(a_0) \leq a_j$ for every $j$.   
\end{proof}

\begin{construction}

Given a functor $F:\mathcal{C}^{op} \to \mathfrak{K}$, with $\mathfrak{K}$ an initial-lax complete order enriched category, then we can define a new functor $\tilde{F}:\mathcal{C}^{op} \to \mathfrak{K}$ by $\tilde{F}(a) = lim^{\sqsubseteq}_{(\mathcal{C}/a)^{op}}((\mathcal{C}/a)^{op} \xrightarrow{\Sigma_a^{op}} \mathcal{C}^{op} \xrightarrow{F} \mathfrak{K}$), with morphisms defined via the universal characterisation of the lax limit. (Recall $\Sigma_a: \mathcal{C} / c \to \mathcal{C}$ is the forgetful functor, and $\mathcal{C}/a$ has a terminal object for every object $a$ of $\mathcal{C}$.) We will use point set notation in what follows as a convenient notation; so, 
\begin{eqnarray*}
\tilde{F}(a) = \{ (x_f)_{f:b \to a} \in \prod_{f: b\to a} F(b)| F(g)x_f \sqsubseteq x_{fg} \text{ } \forall c \xrightarrow{g} b \xrightarrow{f} a \}
\end{eqnarray*}
and for $(x_f) \in \tilde{F}(a')$, and $h:a' \to a$ we define

\[ ([\tilde{F}(h)](x_f))_{f'} =  x_{hf'} \]
for all $f' :b' \to a'$.
  
\end{construction}

\begin{construction}
  Let $F_1, F_2: \mathcal{C}^{op} \to \mathfrak{K}$ be two functors, with $\mathfrak{K}$ an initial-lax complete order enriched category.

  If $\phi : F_1 \xrightarrow{\sqsubseteq} F_2$ is a lax natural transformation, define $\tilde{\phi}:\tilde{F_1} \to \tilde{F_2}$ by

  \[ [\tilde{\phi}_a ((x))]_f = \phi_b (x_f) \]
  for all $f:b \to a$. $\tilde{\phi}_a (x)$ is indeed an element of $\tilde{F_2}$ as for any $ c \xrightarrow{g} b \xrightarrow{f} a$, we have
\[\begin{split} F_2(g) [\tilde{\phi}_a ((x))]_f = & F_2(g) \phi_b(x_f) \\ \sqsubseteq & \phi_c(F_1(g) x_f) \sqsubseteq \phi_c(x_{fg}) = [\tilde{\phi}_a((x))]_{fg} \end{split} \]

  And finally, $\tilde{\phi}$ is a natural transformation, as for any $h:a'\to a$ and $f': b' \to a'$ we have
  \[ \begin{split} [[\tilde{F_2}(h)](\tilde{\phi}_a((x)))]_{f'} =& [\tilde{ \phi}_a ((x))]_{hf'}  \\ =& \phi_{b'}(x_{hf'}) = \phi_{b'}(([\tilde{F}_1(h)](x))_{f'}) = [\tilde{\phi}_{a'}([\tilde{F}_1(h)] (x))]_{f'}  \end{split} \]

  This construction clearly defines an order enriched functor
  \[ \tilde{ ( \_ ) } : [\mathcal{C}^{op},\mathfrak{K}]^{\sqsubseteq} \to  \left[\mathcal{C}^{op},\mathfrak{K} \right]. \]
  
\end{construction}

\begin{lemma}\label{tilde}
Let $\mathfrak{K}$ be an initial-lax complete order enriched category and $F_1, F_2: \mathcal{C}^{op} \to \mathfrak{K}$ two functors (i.e. two presheaves). Then the map $\tilde{( \_)} : Nat^{\sqsubseteq}[F_1,F_2] \to Nat[\tilde{F_1}, \tilde{F_2}]$ is an inclusion with a right adjoint (we which will write $\alpha \mapsto \psi^{\alpha}$). Further, 

(i) $\psi^{Id_{\tilde{F}}} = Id_F$ and $\psi^{ \beta} \psi^{ \alpha } \sqsubseteq \psi^{ \beta \alpha}$; and, 

(ii) $\psi^{\alpha}$ is lax-natural in $\alpha$.
\end{lemma}

\begin{proof}

  We first observe that if $F : \mathcal{C}^{op} \to \mathfrak{K}$ is any presheaf, there are two maps $\mu : \tilde{F} \to F$ and $\epsilon : F \to \tilde{F}$, defined as follows:
  \[
   \epsilon^F_a : \begin{array}{rcl}
      F(a) & \to & \tilde{F}(a) \\
       x & \mapsto & (F(f)x)_{f:b \to a}            
    \end{array} \qquad   \mu^F_a:  \begin{array}{rcl}
      \tilde{F}(a) & \to & F(a)\\
       (y_f)_{f:b \to a} & \mapsto & y_{Id_a}            
    \end{array}
  \]

  One easily checks that $\epsilon$ is a natural transformation (and takes values in $\tilde{F}(a)$) and $\mu$ is a lax natural transformation. Now, given a natural transformation $\alpha:\tilde{F_1} \to \tilde{F_2}$, we define $\psi^{\alpha}:F_1 \xrightarrow{\sqsubseteq} F_2$ as the composite $\psi^\alpha : \mu^{F_2} \circ \alpha \circ \epsilon^{F_1}$; explicitly, 
  \[ \psi^{\alpha}_a(x) = (\alpha_a((F_1(f)x)_{f:b \to a}))_{Id_a} \]

We then observe that
\begin{eqnarray*}
\psi^{\tilde{\phi}}_a(x) & = & ( \tilde{\phi} ( ( F(f)(x))_{f: b \to a}))_{Id_a} \\
 & = & ( (\phi_b F(f)(x))_{f: b \to a}))_{Id_a} \\
 & = & \phi(x)
\end{eqnarray*}
And finally for any natural transformation $\alpha :  \tilde{F_1} \to  \tilde{F_2}$, $\tilde{\psi^{\alpha}}(x_f) = (\psi^{\alpha}_b (x_f))_{f: b \to a}$. But for each $f: b \to a$,
\begin{eqnarray*}
\psi^{\alpha}_b (x_f) & = & ( \alpha_b (( F_1( g)(x_f))_{g: c \to b} )_{Id_b} \\
  & \sqsubseteq & ( \alpha_b ( (x_{fg})_{g: c \to b}) )_{Id_b} \\
 & = &  (\alpha_a( (x_f)_{f:b \to a}) )_f
\end{eqnarray*}
where the last line is by naturality of $\alpha$ at $f$ and the second last line uses that $F_1(g) (x_f) \sqsubseteq x_{fg}$. Hence $\tilde{\psi^{\alpha}} \sqsubseteq \alpha$. Together this shows that $ \tilde{( \_ ) } \dashv \psi^{( \_ )} $ and that $\tilde{( \_ ) } : Nat^{\sqsubseteq}[F_1,F_2] \to Nat[\tilde{F_1}, \tilde{F_2}]$ is injective.

For the `further' part (i), the preservation of identities is immediate from construction, and the inequality is clear as $\epsilon_a^F \mu_a^F \sqsubseteq Id$ by definition of $\tilde{F}(a)$. Part (ii) follows as $\mu$ is natural in $F$ and $\epsilon$ lax natural (explicitly we are asserting that if $\phi_i: F_i \to G_i$, $i = 1,2$ are two lax natural transformations such that $\beta \tilde{\phi_1} = \tilde{\phi_2} \alpha $, then $\psi^{\beta} \phi_1 \sqsubseteq  \phi_2 \psi^{\alpha} $).  
\end{proof}

\begin{remark}
 What is happening here is that $\epsilon$ and $\mu$ defined above are the unit and co-unit of a  KZ-adjunction between $\tilde{(\_)} : [\mathcal{C}^{op},\mathfrak{K}]^{\sqsubseteq} \to [\mathcal{C}^{op},\mathfrak{K}] $ and the forgetful functor $[\mathcal{C}^{op},\mathfrak{K}] \to [\mathcal{C}^{op},\mathfrak{K}]^{\sqsubseteq}$.  The notion of KZ-adjunction is introduced in section 4.1 of \cite{singular} and is sometimes called a \emph{lax-idempotent} adjunction.
\end{remark}

\section{Examples of our `lax to ordinary' functor $\tilde{ (\_)}$ in action}
\label{sec:exemple}

Given the forgoing we must now provide some `real life' examples of $\tilde{(\_)}$ in action. The last one below corresponds to the relative version of the construction $C$ introduced in Proposition \ref{prop:C_completion}.

\begin{example}
Recall the powerset construction, $P:\mathbf{Set} \to \mathbf{Pos}$, which sends a set to its powerset and $f:A \to B$ to $\exists_f: PA \to PB$. This can be done in any topos; given an object $A$ of $[\mathcal{C}^{op}, \mathbf{Set}]$, i.e. a presheaf, consider its powerset $P_{\hat{\mathcal{C}}}A: \mathcal{C}^{op} \to \mathbf{Pos}$. In this example we show that $P_{\hat{\mathcal{C}}}A \cong \widetilde{P \circ A }$. 

We will first assume that $\mathcal{C}$ has a terminal object $1$ and show that $P_{\hat{\mathcal{C}}}A$ evaluated at $1$ is isomorphic to $\widetilde{P \circ A }$ evaluated at $1$. The presheaf $P_{\hat{\mathcal{C}}}A$ evaluated at $1$ is $Sub_{\hat{\mathcal{C}}}(A)$, i.e. the collection of monomorphisms from $I$ to $A$. But a subobject $I \subseteq A$ is a collection of subobjects $I(a) \subseteq A(a)$ such that for any morphism $f : b \to a$ of $\mathcal{C}$, the image of $I(a)$ under $A(f)$ factors through $I(b)$. This is just another way of stating that the morphism $I \subseteq A$ is a natural transformation. As the image of $I(a)$ under $P(f)$ factors through $I(b)$ iff $\exists_{A(f)}(I(a)) \subseteq I(b)$, it follows that $[P_{\hat{\mathcal{C}}}A](1)$ is isomorphic to
\begin{eqnarray*}
\{ (I_a) \in \prod_{a \in Ob(\mathcal{C})} P(A(a)) | \exists_{A(f)}I_a \subseteq I_b \text{ } \forall f: b \to a \in \mathcal{C} \} \text{.} 
\end{eqnarray*}

Given a general object $a$ of $\mathcal{C}$, recall that $[ \mathcal{C}^{op}, \mathbf{Set} ] / \mathcal{C}( \_ , a) \simeq [ (\mathcal{C}/a)^{op} , \mathbf{Set} ]$ so there is a geometric morphism $ \gamma_a: \hat{\mathcal{C} / a}  \to  $$ \hat{\mathcal{C}}$ whose inverse image has a left adjoint (it is a slice, A4.1.3 \cite{Elephant}). The left adjoint sends $1$ to $\mathcal{C}( \_ , a)$ and the inverse image is precomposition with the forgetful functor $\Sigma^{op}_a:(\mathcal{C}/ a) ^{op} \to \mathcal{C}^{op}$ (A4.1.4).  
Now $P_{\hat{\mathcal{C}}}A(a)$ is naturally isomorphic to $Nat[\mathcal{C} ( \_ , a) , P_{\hat{\mathcal{C}}}A] $ which is, via this adjunction, isomorphic to 
\begin{eqnarray*}
Nat[1 , \gamma_a^*P_{\hat{\mathcal{C}}}A] & \cong & Nat[1 , P_{\hat{\mathcal{C}/a}}(A \circ \Sigma_a)]
\end{eqnarray*}

where the isomorphism follows as $\gamma_a$ is logical (all geometric morphisms that are slices are logical; e.g. A2.3.2 \cite{Elephant}) and so its inverse image commutes with the powerset. As $\mathcal{C}/a$ has a terminal object we can apply the above reasoning to conclude that $P_{\hat{\mathcal{C}}}A(a)$ is isomorphic to  
\begin{eqnarray*}
\{ (I_f) \in \prod_{f:b \to a}  P(A(b)) | \exists_{A(g)}I_f \subseteq I_{fg} \text{ } \forall c \xrightarrow{g} b \xrightarrow{f} a \}
\end{eqnarray*}

which is the lax limit of the diagram $(\mathcal{C}/a)^{op} \xrightarrow{\Sigma_a^{op}} \mathcal{C}^{op} \xrightarrow{P \circ A} \mathbf{Pos}$. This establishes order isomorphisms $P_{\hat{\mathcal{C}}}A(a) \cong  \widetilde{P \circ A}(a)$ for every object $a$ of $\mathcal{C}$.

We must also check that these order isomorphisms are natural for any morphism $h: a' \to a$ of $\mathcal{C}$. This essentially follows as $\gamma_{a'} \cong \gamma_h \gamma_a$. The effect of $\gamma_h^*$ on subobjects $ I \subseteq A \circ \Sigma_a^{op}$ is precomposition with the `postcompose with $h$' functor $\Sigma_h^{op}: (\mathcal{C} / a')^{op} \to (\mathcal{C} / a)^{op}$. So $(I_f)_{f:b \to a}$ is mapped to $(I_{hf'})_{f':b' \to a'}$ by $\gamma^*_h$ which is the formula we have given for $\widetilde{P \circ A}(h)$. 

Finally we need to check naturality with respect to a natural transformation $\alpha: A \to B $ (i.e. with respect to a morphism of $\hat{\mathcal{C}}$). That is we must check that 
\[ \begin{tikzcd}
  P_{\hat{\mathcal{C}}} A \ar[r,"\cong"] \ar[d,"\exists_\alpha"] & \widetilde{P \circ A} \ar[d,"\widetilde{P \alpha}"] \\
  P_{\hat{\mathcal{C}}} B \ar[r,"\cong"swap] & \widetilde{P \circ B}    
  \end{tikzcd} \]
commutes. As above this can be seen by first checking the case of the diagram evaluated at $a = 1$ and then applying change of base. The case $a=1$ follows as for any subobject $I \subseteq A$ and for any object $b$ of $\mathcal{C}$, $(\exists_{\alpha}(I))(b) = \exists_{\alpha_b}(I(b))$. (For this last well know fact recall that every object $b$ gives rise to a geometric morphism $p_b:\mathbf{Set} \to \hat{\mathcal{C}}$ whose inverse image is `evaluate at $b$'; existential quantification $\exists$ commutes with inverse images.)
\end{example}

This example is a standard application of basic topos theory, but it was worth writing out the reasoning in full as it generalises:
\begin{example} \label{ex:ideal_completion}
Consider the ideal completion of a poset $idl: \mathbf{Pos} \to \mathbf{Pos}$. Then for any poset in $\hat{\mathcal{C}}$, $idl_{\hat{\mathcal{C}}}P \cong \widetilde{idl \circ P }$. 

Recall that a subset $I \subseteq P$ is an ideal if and only if various geometric sequents are satisfied ($a \leq b \in I \Rightarrow a \in I, \exists * \in I, a,b \in I \Rightarrow \exists c \in I \wedge a,b \leq c$). Because these are geometric, they are preserved by any inverse image functor. That is, if $I \subseteq P$ is an ideal in a topos, then for any geometric morphism $f$, $f^* I \subset f^* P$ is again an idea. In particular, if $I \subseteq P$ is an ideal in the topos $\widehat{\mathcal{C}}$, then given that evaluation at $a \in \mathcal{C}$ is an inverse image functor, $I(a) \subseteq P(a)$ is an ideal in $\mathbf{Set}$ for each $a \in \mathcal{C}$. In fact, $I \subseteq P$ is an ideal if and only if $I(a) \subseteq P(a)$ is an ideal for each object $a$ of $\mathcal{C}$.

As in the last example $I \subseteq P$ iff for any morphism $f : b \to a$ of $\mathcal{C}$, the image of $I(a)$ under $P(f)$ factors through $I(b)$. But for ideals the image of $I(a)$ under $P(f)$ factors through $I(b)$ iff $idl(P(f))(I(a)) \subseteq I(b)$ (recall for any monotone map $f:P_1 \to P_2$ that $idl(f): idl(P_1) \to idl(P_2)$ sends an ideal $I$ of $P_1$ to $\downarrow \{ f(i) | i \in I \}$). It follows that $[idl_{\hat{\mathcal{C}}}P](1)$ is isomorphic to
\begin{eqnarray*}
\{ (I_a) \in \prod_{a \in Ob(\mathcal{C})} idl(P(a)) | idl(P(f))I_a \subseteq I_b \text{ } \forall f: b \to a \in \mathcal{C} \} 
\end{eqnarray*}

from which we establish $idl_{\hat{\mathcal{C}}}P \cong \widetilde{idl \circ P }$, naturally in $P$, just as in the previous example. 
\end{example}

This technique can be applied for any construction that is defined via sets of subsets determined by geometric sequents (provided that the images of the subsets under a presheaf evaluated at a morphism $f:b\to a$ factor iff the constructed morphism (e.g. the $idl(P(f))$ in the last example) evaluated at the domain subset is contain in the codomain subset). In particular:

\begin{example} 
Consider the completion operation $C: \mathbf{NDL} \to \mathbf{NDL}$ introduced in Proposition \ref{prop:C_completion}. Then for any normal distributive lattice $N$ in $\hat{\mathcal{C}}$, we have $C_{\hat{\mathcal{C}}} N \cong \widetilde{ C \circ N }$. 

Indeed, for $N \in \hat{\mathcal{C}}$  an element of $C(N)$ is a subobject of $I \subseteq N$ which satisfies certain geometric axioms, for example ``$i\in I \text{ and } j \leqslant i \Rightarrow j \in I$'' or ``$a \in I \Rightarrow \exists b \in I, a \triangleleft b$'' which can be rewritten as ``$a \in I \Rightarrow \exists b \in I, c \in N, a \wedge c = 0 \text{ and } c \vee b = 1 $''.

Exactly as in Example \ref{ex:ideal_completion}, a subobject of $N$ is a collection of subsets $I(a) \subseteq N(a)$ for each $a \in \mathcal{C}$, such that for $f : b \to a$ in $\mathcal{C}$, the induced map $N(a) \to N(b)$ sends $I(a)$ to $I(b)$. This occurs if and only if $\downarrow \exists_f (I(a)) \subseteq I(b)$. Further, each geometric axiom is satisfied in $\mathcal{C}$ exactly when when it is satsified by $I(a)$ in $N(a)$, for every $a \in \mathcal{C}$. This provides the identification $C_{\hat{\mathcal{C}}} N \cong \widetilde{ C \circ N }$, which is the explicit description needed for the main result. 

\end{example}

\section{Proof of the main theorem}
\label{sec:main_th}

We are now ready to prove the main theorem of the paper. 
\begin{theorem}
For any small category $\mathcal{C}$ there is an equivalence of categories
\begin{eqnarray*}
\mathbf{KRegFrm}_{\hat{\mathcal{C}}} \simeq [\mathcal{C}^{op},\mathbf{KRegFrm}  ] \text{.}
\end{eqnarray*}
\end{theorem}

\begin{proof}
  Firstly, recall that $\mathbf{NDL}_{\hat{\mathcal{C}}}$ is isomorphic to $[\mathcal{C}^{op},\mathbf{NDL}]$; this is implicit in the exposition above as it is the category of models of a geometric theory, or consult D.1.2.14 of \cite{Elephant}. We can therefore treat compact regular frames in $\hat{\mathcal{C}}$, firstly as normal distributive lattices via the forgetful functor and then as functors $\mathcal{C}^{op} \to \mathbf{NDL}$. That is, in what follows we identify $\mathbf{KRegFrm}_{\hat{\mathcal{C}}}$ with the non-full subcategory of $[\mathcal{C}^{op},\mathbf{NDL}]$ which is fixed by the idempotent endofunctor $C_{\widehat{\mathcal{C}}} = \widetilde{ C \circ \_ }$. We define the two functors:
  \[
    \Phi: \begin{array}{ccc} \mathbf{KRegFrm}_{\hat{\mathcal{C}}} &\to& [\mathcal{C}^{op},\mathbf{KRegFrm}] \\
      L & \mapsto & C \circ L
    \end{array}\]\[
    \Psi: \begin{array}{ccc} [\mathcal{C}^{op},\mathbf{KRegFrm}] &\to& \mathbf{KRegFrm}_{\hat{\mathcal{C}}} \\
      A & \mapsto & \widetilde{A}
          \end{array}  \]        
These are well defined: given $L \in \mathbf{KRegFrm}_{\hat{\mathcal{C}}}$, and more generally any $ L\in [\mathcal{C}^{op},\mathbf{NDL}]$, the composite $C \circ L$ determines a functor $\mathcal{C}^{op}  \to \mathbf{KRegFrm}$; and this is clearly natural in $L$. For any $A \in [\mathcal{C}^{op},\mathbf{KRegFrm}]$, we have $A \cong C \circ A$ (functorially), and hence $\widetilde{A} \simeq \widetilde{C \circ A} \simeq C_{\hat{\mathcal{C}}} (A)$, hence $\widetilde{A}$ being of the form $C_{\hat{\mathcal{C}}} ( \_)$ it is a compact regular frame in $\widehat{\mathcal{C}}$ by an internal application of Proposition \ref{prop:C_compact}.

Certainly $ L \cong C_{\hat{\mathcal{C}}} L \cong \widetilde{ C \circ L }$, from which $\Psi \Phi \cong Id_{\mathbf{KRegFrm}_{\hat{\mathcal{C}}}}$.  

So we have to but check $\Phi \Psi \cong Id_{[\mathcal{C}^{op},\mathbf{KRegFrm}  ]}$; that is, that $C \circ \widetilde{ A} \cong A$ naturally for each $A:\mathcal{C}^{op} \to \mathbf{KRegFrm}$. Since $\widetilde{ A}$ is a compact regular frame in $\widehat{\mathcal{C}}$ we know that there is an isomorphism $\alpha: \widetilde{ C \circ \widetilde{A} } \to \widetilde{A}$; this gives rise to a lax natural transformation $\psi^{\alpha}: C \circ \widetilde{A} \to A$, using the notation of Lemma \ref{tilde}. But then note that for any object $a$ of $\mathcal{C}$,  $\psi_a^{\alpha}: [C \circ \widetilde{A} ](a) \to A(a)$ is a suplattice homomorphism as well as a lattice homomorphism; i.e. it is a frame homomorphism. (Recall Example \ref{sup_example}; so the $\epsilon$ and $\mu$ used to construct $\psi^{\alpha}$ are both also suplattice homomorphisms and certainly $\alpha_a$ is a suplattice homomorphism as it is an order-isomorphism.) But then $\psi_a^{\alpha}$ is a frame homorphism between compact regular frames, so we can use $\psi^{Id_{\tilde{F}}} = Id_F$ and $\psi^{ \beta} \psi^{ \alpha } \sqsubseteq \psi^{ \beta \alpha}$ established in Lemma \ref{tilde} to see that  $\psi_a^{\alpha}$ is an isomorphism; this is because the partial ordering of frame homomorphisms between compact regular frames is discrete (e.g. III Lemma 1.5 of \cite{Stone}). Next, $\psi^{\alpha}$ is constructed as a lax natural transformation, but in fact it can be seen to be a natural transformation as the homsets of $\mathbf{KRegFrm}$ are discrete. Finally, for naturality in $A$, recall part (ii) of Lemma \ref{tilde}; again we only have lax-naturality from the Lemma, but this is sufficient as the morphisms involved are all frame homomorphisms between compact regular frames.  
 
(Note in the above that relative to both our base topos $\mathbf{Set}$ and $\hat{\mathcal{C}}$ we are passing through a forgetful functor back to $\mathbf{NDL}$ without notation; however, these forgetful functors create isomorphisms.) 
\end{proof}

\begin{remark}
It should be noted that Theorem 25 of \cite{gelfand} also in effect provides this description of compact regular frames in a presheaf topos 
\end{remark}

\begin{remark} It should be noted that, Joyal and Tierney gave in \cite{extension} a general description of internal frames in the topos $\widehat{\mathcal{C}}$, at least in the case where $\mathcal{C}$ has finite limits, which in theory could be specialized to a description of compact regular frame in $\widehat{\mathcal{C}}$. However, this description works differently from ours, and would not recover our result, but we can explain how they relates.

    Each object of $a \in \mathcal{C}$ induces a geometric morphisms $a: \text{Set} \to \widehat{\mathcal{C}}$ defined by $a^* F = F(a)$. Now given a frame (or locale) $L$ in $\widehat{\mathcal{C}}$ there are two different ways to use this point to get a frame in Set: In general the image of frame by an inverse image functor will not be a frame, but in this specific case we do obtain a frame $a^* L = L(a)$. There is however a way to ``pullback'' a frame along a geometric morphism, by taking a site of definition for the frame and pullback the site using $a^*$, the resulting construction on frames is sometime denoted $a^\# L$, it also corresponds to the pullback of toposes
   \[ \begin{tikzcd}[ampersand replacement = \&]
      \text{Sh}(a^\# L) \ar[r] \ar[d] \ar[dr,phantom,"\lrcorner"very near start] \& \text{Sh}_{\widehat{\mathcal{C}}}(L) \ar[d] \\
      \text{Set} \ar[r,"a"] \& \widehat{\mathcal{C}}
    \end{tikzcd}\]

  Now, in general these $a^\# L$ are not functorial in the point $a$ (at least not for the usual notion of morphisms of frame or locale). However, in the special case where $L$ is a compact Hausdorff frame it is possible to show that $a^\# L = C(a^* L)$, which proves that they are covariant in $a \in \mathcal{C}$ when seen as values in frame.

  \bigskip

  Now, the key point is that the Joyal-Tierney description of frames in $\widehat{\mathcal{C}}$ is in terms of the functor $L(a) = a^* L$, while ours is in terms of the functors $L_a = a^\#  L = C(a^* L)$. In particular, they are connected by the formulas:
  \[ L_a = C(L(a)) \qquad L(a) = \widetilde{L_a} \]
 Here the second identity, follows from the observation above that for $L$ an internal compact regular frame seen as an object in $\mathbf{NDL}_{\widehat{\mathcal{C}}} = [\mathcal{C}^{op},\mathbf{NDL}]$ satisfies $L = C_{\widehat{\mathcal{C}}} L = \widetilde{C \circ L}$, so as $L_a = C(L(a))$ one obtains that $L(a) = \widetilde{C(L(a))} = \widetilde{L_a}$. The ``$L(a)$'' construction also corresponds with the identification of $\mathbf{KRegFrm}_{\hat{\mathcal{C}}}$ as a non-full subcategory of $[\mathcal{C},\mathbf{NDL}]$ that we used during the proof.

  We do not expect that such a description in terms of the $L_a$ as we provided here for compact regular frame can be extended to general frame.

\end{remark}

\end{document}